\documentclass[11pt]{article}
\usepackage{amssymb, amsmath, amsthm, amsfonts,xcolor, enumerate}
\usepackage{tikz}
\usetikzlibrary{patterns}
\usetikzlibrary{shapes}
\usepackage{fullpage}
\usepackage{hyperref}
\newtheorem{theorem}{Theorem}[section]
\newtheorem{lemma}[theorem]{Lemma}

\newtheorem{cor}[theorem]{Corollary}
\newtheorem{question}[theorem]{Question}
\newtheorem{conj}[theorem]{Conjecture}

\newtheorem{prop}[theorem]{Proposition}
\theoremstyle{definition}

\newcommand{\eps}{\varepsilon}
\newcommand{\om}{\omega}

\newcommand{\cA}{\mathcal{A}}
\newcommand{\cB}{\mathcal{B}}

\newcommand{\cH}{\mathcal{H}}

\newcommand{\cF}{\mathcal{F}}
\newcommand{\cG}{\mathcal{G}}

\newcommand{\card}[1]{\left| #1 \right|}
\newcommand{\floor}[1]{\left \lfloor #1 \right \rfloor}
\newcommand{\ceil}[1]{\left \lceil #1 \right \rceil}
\newcommand{\rt}{\right}
\newcommand{\lt}{\left}

\title{Comparable pairs in families of sets}
\author{
Noga Alon
\thanks{Sackler School of Mathematics and Blavatnik School of Computer Science, Tel Aviv University, Tel Aviv 69978, Israel, and School of Mathematics, Institute for Advanced Study, Princeton, NJ 08540, USA.  {\tt nogaa@tau.ac.ul}.  Research supported in part by a USA-Israeli BSF grant, by an ISF grant, by the Israeli I-Core program and by the Oswald Veblen Fund.}
\and
Shagnik Das
\thanks{Department of Mathematics, Freie Universit\"at, Berlin, Germany. \tt{shagnik@mi.fu-berlin.de}.}
\and
Roman Glebov
\thanks{Department of Mathematics, ETH, 8092 Zurich, Switzerland. \tt{roman.glebov@math.ethz.ch}.}
\and
Benny Sudakov
\thanks{Department of Mathematics, ETH, 8092 Zurich, Switzerland. {\tt benjamin.sudakov@math.ethz.ch}. Research supported in part by SNSF grant 200021-149111 and by a USA-Israel BSF grant.}
}

\begin{document}
\maketitle

\begin{abstract}
Given a family $\cF$ of subsets of $[n]$, we say two sets $A, B \in \cF$ are \emph{comparable} if $A \subset B$ or $B \subset A$.  Sperner's celebrated theorem gives the size of the largest family without any comparable pairs.  This result was later generalised by Kleitman, who gave the minimum number of comparable pairs appearing in families of a given size.

In this paper we study a complementary problem posed by Erd\H{o}s and Daykin and Frankl in the early '80s.  They asked for the maximum number of comparable pairs that can appear in a family of $m$ subsets of $[n]$, a quantity we denote by $c(n,m)$.  We first resolve an old conjecture of Alon and Frankl, showing that $c(n,m) = o(m^2)$ when $m = n^{\omega(1)} 2^{n/2}$.  We also obtain more accurate bounds for $c(n,m)$ for sparse and dense families, characterise the extremal constructions for certain values of $m$, and sharpen some other known results.
\end{abstract}

\section{Introduction} \label{sec:intro}

Extremal set theory, with its connections and applications to numerous other fields, has enjoyed tremendous growth in the last few decades.  However, its origins date back much further, with many considering the classic theorem of Sperner~\cite{spe28} to be the starting point of the field.  An \emph{antichain} is a family of pairwise incomparable sets, and Sperner's theorem states that the largest antichain over $[n]$ has $\binom{n}{\floor{n/2}}$ sets.  This bound is attained by the middle levels of the Boolean hypercube; that is, the families of sets of size $\floor{n/2}$ or of sets of size $\ceil{n/2}$.

\subsection{Comparable pairs}

Given this extremal result, it is natural to ask how many comparable pairs may appear in set families over $[n]$ of a given size.  For the minimisation problem, Sperner's theorem shows that there need not be any comparable pairs in families of size $m \le \binom{n}{\floor{n/2}}$.  Kleitman~\cite{kle66} later completely resolved the problem for larger families, showing there is a nested sequence of extremal families consisting of sets as close to the middle levels as possible.  He further conjectured that the same families also minimise the number of \emph{$k$-chains}, which are collections of $k$ pairwise-comparable sets $F_1 \subset F_2 \subset \hdots \subset F_k$, for every $k$.  Erd\H{o}s~\cite{erd45} had earlier generalised Sperner's theorem to show that the largest $k$-chain-free family consisted of the $k-1$ middle levels of the hypercube.  Das, Gan and Sudakov~\cite{dgs14} verified Kleitman's conjecture for families of size up to the $k+1$ middle levels of the hypercube (Dove, Griggs, Kang and Sereni~\cite{dgks14} independently obtained the same result for $k$ middle levels), but the conjecture otherwise remains open.

Our focus is the corresponding maximisation question, a problem which has also attracted a great deal of attention.  To this end, we denote by $c(\cF)$ the number of comparable pairs in the set family $\cF$, and define $c(n,m)$ to be the maximum number of comparable pairs in a family of $m$ subsets of $[n]$.  We clearly have $c(n,m) \le \binom{m}{2}$, and can only have equality for an $m$-chain, which requires $m \le n+1$.  Daykin and Frankl~\cite{df83} showed that one may have much larger families where \emph{almost all} pairs are comparable.  More precisely, they proved $c(n,m) = (1 - o(1)) \binom{m}{2}$ if and only if $m = 2^{o(n)}$.

Their lower bound came from a construction known as a \emph{tower of
cubes}, which generalises chains.  A \emph{subcube} of the hypercube
$2^{[n]}$ is, for some $F_1 \subset F_2$, the family of sets $\{ F
\subset [n] : F_1 \subset F \subset F_2 \}$.  We say the dimension
of
the subcube is $\card{F_2} - \card{F_1}$. To define a tower of cubes,
assume for simplicity $n$ is divisible by $k$, and let $\ell = n/k$.
Let $X_i = [i\ell]$ for $1 \le i \le k$, and consider the subcube $\cF_i
= \left\{ F \subset [n] : X_{i-1} \subseteq F \subseteq X_i \right\}$.
A tower of $k$ cubes is the set family $\cF = \bigcup_{i = 1}^{k} \cF_i$.
We have $m = \card{\cF} = k 2^{n/k} - k + 1$.  Moreover, two sets from
different subcubes must be comparable, and so we have $c(\cF) \ge \left(
1 - \frac{1}{k} \right) \binom{m}{2}$.  When $k = \om(1)$, we have $c(\cF)
= (1 - o(1)) \binom{m}{2}$, as in the theorem of Daykin and Frankl.

Alon and Frankl~\cite{af85} proved the towers of cubes are asymptotically optimal even when $k$ is constant, as shown by their theorem below.

\begin{theorem} \label{thm:af}
For every positive integer $k$ there exists a positive $\beta = \beta(k)$ such that if $m = 2^{(1/(k+1) + \delta)n}$ for $\delta > 0$, then
\[ c(n,m) < \left( 1 - \frac{1}{k} \right) \binom{m}{2} + O \left( m^{2 - \beta \delta^{k+1}} \right). \]
\end{theorem}

The case $k=1$ is of particular interest.  In~\cite{af85}, Alon and Frankl prove that a family $\cF$ of size $m = 2^{(1/2 + \delta)n}$ must have $c(\cF) < 4 m^{2 - \delta^2 / 2} = o(m^2)$, thus proving a conjecture of Daykin and Erd\H{o}s~\cite{guy83}.  Erd\H{o}s~\cite{erd81} had made a finer conjecture, asking whether $m = \om(2^{n/2})$ implies $c(n,m) = o(m^2)$.  Alon and Frankl disproved this conjecture, exhibiting for any $d \ge 1$ a family $\cF$ of size $\Omega(n^d 2^{n/2})$ with $c(\cF) \ge 2^{-2d-1} \binom{m}{2}$.  They in turn conjectured that this construction, which we describe in detail in Section~\ref{sec:afconj}, is essentially the best possible.

\begin{conj} \label{conj:af}
If $m = n^{\om(1)} 2^{n/2}$, then $c(n,m) = o(m^2)$.
\end{conj}

\subsection{Our results}

In this paper, we further the study of the maximum number of comparable pairs in a set family.  We begin by proving Conjecture~\ref{conj:af}.  We shall in fact prove the slightly more general two-family version below.  Given two set families $\cA$ and $\cB$, we write $c(\cA, \cB)$ for the number of pairs $(A,B) \in \cA \times \cB$ with $A \subset B$.  

\begin{theorem} \label{thm:afconj}
If $\cA$ and $\cB$ are set families over $[n]$ with $\card{\cA} \card{\cB} = n^d 2^n$, then $c(\cA,\cB) \le 2^{- d/300} \card{\cA} \card{\cB}$.
\end{theorem}

Conjecture~\ref{conj:af} follows easily, since for any set family $\cF$ with $\card{\cF} = m = n^{\om(1)} 2^{n/2}$, Theorem~\ref{thm:afconj} implies $c(\cF) = c(\cF, \cF) \le 2^{-\om(1)} \card{\cF}^2 = o(m^2)$.

We next strengthen Theorem~\ref{thm:af} by proving a stability result, showing families with close to $(1 - 1/k) \binom{m}{2}$ comparable pairs must be close in structure to a tower of $k$ cubes.

\begin{theorem} \label{thm:afstability}
For every $\eps > 0$ and integer $k \ge 2$, there is an $\eta > 0$ such that for sufficiently large $n$, if a set family $\cF$ over $[n]$ of size $m = \card{\cF} \ge (1 - \eta) k 2^{n/k}$ has at least $\left( 1 - \frac{1 + \eta}{k} \right) \binom{m}{2}$ comparable pairs, then all but at most $\eps m$ sets in $\cF$ are contained inside a tower of $k$ cubes of dimension $n/k$.
\end{theorem}

As a consequence, we are able to show that the towers of cubes are (uniquely) extremal.

\begin{cor} \label{cor:exact}
Given $k \ge 2$ and $n$ sufficiently large, if $k | n$ and $\cF$ is a set family over $[n]$ of size $m = \card{\cF} = k2^{n/k} - k +1$ maximising the number of comparable pairs, then $\cF$ is a tower of $k$ cubes of dimension $n/k$.
\end{cor}

For sparser families, when $m = 2^{o(n)}$, we refine the theorem of Daykin and Frankl by using Theorem~\ref{thm:af} to determine the value of $c(n,m)$ more precisely.  Since almost all pairs are comparable in this regime, we instead count the number of incomparable pairs in a family $\cF$, which we denote by $i(\cF)$.  We denote by $i(n,m)$ the minimum number of incomparable pairs in a family $\cF$ of $m$ subsets of $[n]$.  Daykin and Frankl showed we have $i(n,m) = o( m^2 )$ in this range, and we provide asymptotic lower bounds.  In this setting, it is more convenient to parameterise $m = \card{\cF} = n \ell$.

\begin{theorem} \label{thm:sparse}
Given $\eps > 0$, for $\ell$ and $n$ sufficiently large we have $i(n,n\ell) \ge \left(1/2 - \eps \right) n \ell^2 \log \ell$.
\end{theorem}

The above theorem shows that in the sparse regime, the towers of cubes are again asymptotically optimal in this finer sense.  Indeed, a tower of $k$ cubes has $\card{\cF} = n \ell = k2^{n/k} - k + 1$ and $i(\cF) \approx \frac{1}{k} \binom{n \ell}{2}$, since almost all pairs of sets from the same subcube are incomparable.  This shows that Theorem~\ref{thm:sparse} is asymptotically tight, as $n \ell^2 \log \ell / 2 \approx \frac{\log \ell}{n} \binom{n \ell}{2} \approx \frac{1}{k} \binom{n \ell}{2}$.

The above results all hold for very sparse families, where $m = \card{\cF}$ is much smaller than $2^n$, the total number of subsets of $[n]$.  We finally turn our attention to dense families, where we show the extremal families are of a very different nature.  Note that, within the complete hypercube $2^{[n]}$, the sets of size $n/2$ are in the fewest comparable pairs, while the sets of extreme size, namely $\emptyset$ and $[n]$, are in the most.  Intuitively, we might expect that families maximising the number of comparable pairs should avoid sets near the middle levels, and using shifting arguments we are able to prove this is the case.  Given $0 \le k \le n/2$, define $\cH_k = \{ F \subset [n] : \card{F} \le k \} \cup \{ F \subset [n]: \card{F} \ge n-k \}$, and let $M_k = \card{\cH_k} = 2 \sum_{i \le k} \binom{n}{i}$.

\begin{theorem} \label{thm:dense}
If $M_{k-1} \le m \le M_k$ for some $k$ with $n/3 + \sqrt{2 n\ln 2} \le k \le n/2$, then every family $\cF$ of $m$ sets over $[n]$ maximising the number of comparable pairs satisfies $\cH_{k-1} \subset \cF \subset \cH_k$.
\end{theorem}

The entropy bound (see Lemma~\ref{lem:entropy}) gives the estimate $\sum_{i \le pn} \binom{n}{i} \le 2^{H(p) n }$, where $H(p) = -p \log p - (1-p) \log (1-p)$.  It follows that Theorem~\ref{thm:dense} applies when $m \ge 2^{0.92n}$, giving the large-scale structure of extremal families in this range.  To determine $c(n,m)$ precisely, one must prescribe which sets of size $k$ and $n-k$ should be chosen.  We are able to do so for some special values of $m$, which we discuss further in Section~\ref{sec:dense}.

\subsection{Notation and organisation}

For a family of sets $\cF$ and $x \in [n]$, we define $\cF(x) = \{ F \in \cF : x \in F \}$.  We say an element $x$ is $\eta$-\emph{dense} in $\cF$ if $\card{\cF(x)} \ge \eta \card{\cF}$; that is, it is contained in at least an $\eta$-proportion of sets in $\cF$.  We say $x$ is \emph{covered} by $\cF$ if it is contained in some set in $\cF$, i.e. $\card{\cF(x)} \ge 1$.  The remainder of our set notation is standard.  All asymptotics are taken as $n \rightarrow \infty$.  We use $\log$ for the binary logarithm and $\ln$ for the natural logarithm.

In Section~\ref{sec:afconj}, we prove Theorem~\ref{thm:afconj}, settling Conjecture~\ref{conj:af}.  We next extend the known results for sparse families in Section~\ref{sec:sparse}, proving Theorems~\ref{thm:afstability} and~\ref{thm:sparse}.  In Section~\ref{sec:dense} we turn to dense families, proving Theorem~\ref{thm:dense}.  Finally, Section~\ref{sec:conc} contains some concluding remarks and open questions.

\section{The Alon--Frankl Conjecture} \label{sec:afconj}

In this section we will prove Theorem~\ref{thm:afconj}, thus resolving Conjecture~\ref{conj:af}.  We begin, however, by describing the construction given in~\cite{af85}, which motivates some of the ideas behind the proof.

Recall that our goal is to construct, for every constant $d$, a family $\cF$ of size $\Omega\left(n^d 2^{n/2}\right)$ with a positive density of comparable pairs.  For simplicity, we shall assume $n$ is even.  When $d = 0$, the natural construction is to take a tower of two cubes, namely all subsets of $[n/2]$ and all sets containing $[n/2]$.  Since any set of the first type is comparable to any set of the second, this gives a family $\cF$ of $2 \cdot 2^{n/2}$ sets with at least $\frac12 \binom{\card{\cF}}{2}$ comparable pairs.

To obtain larger families, we `fatten' the construction.  Let $\cF = \cF_1 \cup \cF_2$, where $\cF_1 = \{ F \subset [n] : \card{F \setminus [n/2]} \le d \}$ and $\cF_2 = \{ F \subset [n] : \card{[n/2] \setminus F} \le d \}$, so that sets of the first type are now allowed to have at most $d$ elements outside $[n/2]$, while sets of the second type can miss up to $d$ elements from $[n/2]$.  This clearly gives a family of size $\card{F} = \Omega(n^d 2^{n/2})$, and it is not difficult to check that $c(\cF_1, \cF_2) \ge 2^{-2d -1} \binom{\card{F}}{2}$.

In this construction, the elements of $[n]$ are of two types.  Those in $[n/2]$ are in half the sets from $\cF_1$ and almost all the sets from $\cF_2$, while those in $[n] \setminus [n/2]$ are in very few of the sets from $\cF_1$ and in half the sets from $\cF_2$.  Our proof below, for the two-family version of the problem, is guided by an attempt to recover this partition of the ground set in an extremal construction; this motivation is obvious in the proof of Lemma~\ref{lem:smalld}, but such a partition is also a by-product of the calculations in Lemma~\ref{lem:calculus}.

\begin{proof}[Proof of Theorem~\ref{thm:afconj}]
Our proof is by induction on $n$.  It is easy to see the statement holds for $n=2$, by checking it for all choices of $\cA$ and $\cB$.\footnote{Indeed, we may assume $\cA$ is a left-compressed down-set and $\cB$ is a left-compressed up-set, and so we only need to check the families $\cA \in \left\{ \{ \emptyset\}, \{ \emptyset, \{ 1 \} \}, \{ \emptyset, \{ 1 \}, \{2\} \}, \{\emptyset, \{1\}, \{2\}, \{1,2\} \} \right\}$ and $\cB \in \left\{ \{\{1,2\}\}, \{ \{1\}, \{1,2\} \}, \{ \{1\}, \{2\}, \{1,2\} \}, \{ \emptyset, \{1\}, \{2\}, \{1,2\} \} \right\}$.}

We now proceed to the induction step with $n \ge 3$.  Note that the statement is trivial for $d \le 0$, since we always have $c(\cA, \cB) \le \card{\cA}  \card{\cB}$.  Now consider the case $d \in (0,1]$, and let $\cA$ and $\cB$ be two families with $c(\cA, \cB) \ge 2^{ - d / 300} \card{\cA}  \card{\cB}$.  We have $2^{-d/300} \ge 2^{-1/300} > 0.99$, so the comparable pairs are very dense indeed.  In this case, the following lemma, to be proven later, gives the required bound on $\card{\cA}  \card{\cB}$.

\begin{lemma} \label{lem:smalld}
For $n \ge 2$ and $d \in (0,1]$, let $\cA$ and $\cB$ be two families of subsets of $[n]$ satisfying $c(\cA,\cB) \ge 2^{-d/300} \card{\cA}  \card{\cB}$.  Then $\card{\cA}  \card{\cB} < n^d 2^n$.
\end{lemma}

Hence we may assume $d \ge 1$.  In this case, we shall project the families onto $[n-1]$ and apply induction.  Let $\cA_0 = \{ A \in \cA : n \notin A \}$ and $\cA_1 = \{ A \subset [n-1] : A \cup \{n\} \in \cA \}$, and define $\cB_0$ and $\cB_1$ similarly; note that these families are supported on $[n-1]$.  We have
\begin{equation} \label{eqn:comppairssum}
c(\cA,\cB) = c(\cA_0, \cB_0) + c(\cA_0, \cB_1) + c(\cA_1, \cB_1).
\end{equation}

Now define $p = \card{\cA_0}/\card{\cA}$ to be the probability that $n$ is \emph{not} in a random set of $\cA$, and $q = \card{\cB_1} / \card{\cB}$ to be the probability that $n$ \emph{is} in a random set of $\cB$.  We then have, for instance, $\card{\cA_0}  \card{\cB_0} = p(1-q) \card{\cA} \card{\cB} = p(1-q) n^d 2^n$.  However, to apply the induction hypothesis, we wish to rewrite this in terms of $n-1$.  Define $d_{00}$ to be such that $\card{\cA_0}  \card{\cB_0} = p(1-q) n^d 2^n = (n-1)^{d_{00}} 2^{n-1}$.  A simple calculation reveals $d_{00} = d + \log_{n-1} \left( 2 p(1-q) \left( 1 + \frac{1}{n-1} \right)^d \right)$.  Applying the induction hypothesis, we deduce
\[ c(\cA_0, \cB_0) \le 2^{- \frac{d_{00}}{300}} \card{\cA_0} \card{\cB_0} = p(1-q) \left( 2 p(1-q) \left( 1 + \frac{1}{n-1} \right)^d \right)^{- \frac{1}{300 \log (n-1)}} 2^{-\frac{d}{300}} \card{\cA}  \card{\cB}. \]

Performing similar calculations for the other two terms and then substituting into~\eqref{eqn:comppairssum}, we have
\[ c(\cA, \cB) \le \left[ \left(p(1-q)\right)^{1 - \alpha} + \left(pq\right)^{1 - \alpha} + \left((1-p)q\right)^{1-\alpha} \right] \left( 2 \left( 1 + \frac{1}{n-1} \right)^d \right)^{- \alpha} 2^{- \frac{d}{300}} \card{\cA}  \card{\cB}, \]
where $\alpha = 1/(300 \log(n-1))$.  Note that $n-1 = 2^{1/(300 \alpha)}$, and, since $d \ge 1$, we have
\begin{equation} \label{eqn:comppairsbound}
c(\cA, \cB) \le \left[ \left(p(1-q)\right)^{1 - \alpha} + \left(pq\right)^{1 - \alpha} + \left((1-p)q\right)^{1-\alpha} \right] \left( 2 + 2^{1 - \frac{1}{300 \alpha}} \right)^{- \alpha} 2^{- \frac{d}{300}} \card{\cA}  \card{\cB}.
\end{equation}

We shall later prove the following analytic inequality.

\begin{lemma} \label{lem:calculus}
For $p, q \in [0,1]$ and $\alpha > 0$, we have
\begin{equation} \label{eqn:calculus}
\left( p(1-q) \right)^{1 - \alpha} + \left(pq\right)^{1-\alpha} + \left( (1-p)q \right)^{1- \alpha} \le \left(2 + 2^{1 - \frac{1}{300\alpha}} \right)^{\alpha}.
\end{equation}
\end{lemma}

This completes the induction, since substituting~\eqref{eqn:calculus} into~\eqref{eqn:comppairsbound} gives $c(\cA, \cB) \le 2^{-d/300} \card{\cA}  \card{\cB}$, as required.
\end{proof}

We now give the proofs of the two lemmas.

\begin{proof}[Proof of Lemma~\ref{lem:smalld}]
Let $\eps = 1 - 2^{-d/300}$, so that we have $c(\cA,\cB) \ge (1-\eps) \card{\cA} \card{\cB}$.  We now pass to subfamilies with every set in many comparable pairs.  Let $\cA' = \left\{ A \in \cA : c(\{A\}, \cB) \ge \frac{19}{20} \card{\cB} \right\}$ and $\cB' = \left\{ B \in \cB : c(\cA, \{B \}) \ge \frac{19}{20} \card{\cA} \right\}$.  Given the number of comparable pairs between $\cA$ and $\cB$, it follows from Markov's Inequality that $\card{\cA'} \ge (1 - 20\eps) \card{\cA}$ and $\card{\cB'} \ge (1 - 20\eps) \card{\cB}$.  Since
\[ \frac{19}{20} - 20 \eps \ge \frac{19}{20} - 20 \left(1 - 2^{-\frac{1}{300}} \right) > \frac{9}{10}, \]
for every $A \in \cA'$ we have $c(\{A\}, \cB') \ge c(\{A\},\cB) - \card{\cB \setminus \cB'} \ge \left( \frac{19}{20} - 20 \eps \right) \card{\cB} > \frac{9}{10} \card{\cB'}$, and similarly for every $B \in \cB'$.

We now partition the elements of $[n]$ according to how often they appear in $\cA'$.  Let $[n] = R \cup S \cup T$, where $R$ is the set of elements that are $(1/10)$-dense in $\cA'$, $S$ is the set of elements in $[n] \setminus R$ covered by $\cA'$, and $T = [n] \setminus (R \cup S)$ is the set of elements not covered by $\cA'$.  We use the following result, which appears as Corollary 15.7.3 in~\cite{as}, to bound the size of $\cA'$.

\begin{lemma} \label{lem:entropy}
Let $\cF$ be a family of subsets of $\{1,2,\hdots, n\}$ and let $p_i$ denote the fraction of sets in $\cF$ that contain $i$.  Then 
\[ \card{\cF} \le 2^{\sum_{i=1}^n H(p_i)}, \]
where $H(p) = -p \log p - (1-p) \log (1-p)$.
\end{lemma}

Note that the binary entropy $H(p)$ increases from $0$ to $1$ as $p$ ranges from $0$ to $1/2$, and then decreases back to $0$ as $p$ increases to $1$.  In our setting, we have $p_i \le 1/10$ for $i \in S$ and $p_i = 0$ for $i \in T$, while we could have $p_i = 1/2$ for $i \in R$.  Thus $\card{\cA'} \le 2^{\card{R} + H(1/10) \card{S}}$.

We now seek to bound $\card{\cB'}$.  By our minimum degree condition, every set in $\cB'$ contains at least $9\card{\cA'}/10$ sets in $\cA'$, and hence must contain any element that is $(1/10)$-dense in $\cA'$.  Thus $R \subset B$ for all $B \in \cB'$.  Since every set in $\cA'$ is contained in at least $9 \card{\cB'}/10$ sets in $\cB'$, it follows that any element covered by $\cA'$, and in particular those in $S$, must be $(9/10)$-dense in $\cB'$.  Thus, applying Lemma~\ref{lem:entropy} once more, we have $\card{\cB'} \le 2^{H(9/10) \card{S} + \card{T}}$.

Thus we have $\card{\cA'} \card{\cB'} \le 2^{\card{R} + (H(1/10) + H(9/10))\card{S} + \card{T}} = 2^{n - (1 - 2 H(1/10)) \card{S}} \le 2^n$, since $H(1/10) < 1/2$.  Hence we conclude
\[ \card{\cA} \card{\cB} \le (1 - 20 \eps)^{-2} \card{\cA'} \card{\cB'} < (1 - \eps)^{-300} 2^n = 2^d 2^n \le n^d 2^n, \]
as claimed.
\end{proof}

We conclude this section with the proof of the second lemma.  Here the aforementioned partition of the ground set $[n]$ is not immediately evident.  However, the following calculations show that the left-hand side of~\eqref{eqn:calculus} is maximised when one of $p$ or $q$ is close to $1/2$ while the other is close to $1$.  This implies that in an optimal construction, elements should essentially be fixed in one of $\cA$ or $\cB$ and free in the other, giving a partition as in the construction of Alon and Frankl.

\begin{proof}[Proof of Lemma~\ref{lem:calculus}]
The inequality clearly holds on the boundary of the square, as it takes a maximum of $1$ when $p = 0$ or $q = 0$ and a maximum of $2^{\alpha}$ when $p = 1$ or $q=1$.  Hence we only need to check the local maxima within the square.

Taking the derivative of the left-hand side of~\eqref{eqn:calculus} with respect to $p$, we see that a local maximum can only be obtained at
\begin{equation} \label{eqn:optimalp}
p_0 = p_0(q) = 1 - \frac{q^{\frac{1}{\alpha} - 1}}{ \left( q^{1-\alpha} + (1-q)^{1- \alpha} \right)^{\frac{1}{\alpha}} + q^{\frac{1}{\alpha} - 1}},
\end{equation}
and the value of the left-hand side there is
\[ \left( \left( q^{1-\alpha} + (1-q)^{1- \alpha} \right)^{\frac{1}{\alpha}} + q^{\frac{1}{\alpha} - 1} \right)^{\alpha}. \]

Thus the inequality in the lemma reduces to showing
\begin{equation} \label{eqn:onevarcalc}
f(q) := \left( q^{1-\alpha} + (1-q)^{1-\alpha} \right)^{\frac{1}{\alpha}} + q^{\frac{1}{\alpha} - 1} < 2 + 2^{1 - \frac{1}{300 \alpha}}.
\end{equation}

We first limit the range of $q$.  As the left-hand side of~\eqref{eqn:calculus} is symmetric in $p$ and $q$, we may restrict ourselves to local maxima $(p,q)$ with $q \le p = p_0(q)$.  Suppose we had $q^{\frac{1}{\alpha} - 1} > 1/10$.  Since $\alpha = 1/(300 \log(n-1)) \le 1/300$, we have $1/10 < q^{\frac{1}{\alpha} - 1} \le q^{299}$, which implies $q > 0.99$.  On the other hand, since $(q^{1-\alpha} + (1-q)^{1-\alpha})^{\frac{1}{\alpha}} \le 2$ and $q^{\frac{1}{\alpha} -1} \le 1$,~\eqref{eqn:optimalp} implies $p_0 < 1 - 1/30 < 0.99 < q$, a contradiction.  Hence we may assume $q^{\frac{1}{\alpha} - 1} \le 1/10$.

Define $g(q) = q^{1- \alpha} + (1 - q)^{1-\alpha}$ and $h(q) = g(q)^{\frac{1}{\alpha}} = \left( q^{1- \alpha} + (1-q)^{1 - \alpha} \right)^{\frac{1}{\alpha}}$.  As $g'(q)=(1-\alpha)\left(q^{-\alpha}-(1-q)^{-\alpha}\right)$, $g$ is increasing in $[0,1/2]$ and  decreasing in $[1/2,1]$, and hence so is $h(q)$.  Thus the maximum of $h(q)$ is at $q=1/2$, where it takes the value $2$. Therefore $f(q) \leq 2+q^{\frac{1}{\alpha}-1}$ and the desired inequality holds unless $q^{\frac{1}{\alpha}-1} > 2^{1-\frac{1}{300\alpha}}$,
which comfortably implies $q \geq 0.99$. Thus we may assume $0.99 \leq q \leq p_0$. To complete the proof we show that in this range $f(q) <2$. Given our bound on $q^{\frac{1}{\alpha}-1}$, it suffices to show that $h(q) <1.9$ in this range. Now $h(q)$ is decreasing in
$[1/2,1]$, so it suffices to show that $h(0.99)<1.9$. By the Mean Value Theorem, $h(0.99)=h(0.75)+0.24 h'(q_1)$  for some $q_1 \in [0.75,0.99]$. However, since $g(q) \geq 1$ and $g'(q)$ is negative in this range, 
\begin{align*}
h'(q_1) &=\frac{1}{\alpha} g(q_1)^{\frac{1}{\alpha}-1} g'(q_1) \leq \frac{1}{\alpha} (1-\alpha) (q_1^{-\alpha}-(1-q_1)^{-\alpha}) \\
&=\frac{1}{\alpha} (1-\alpha)(-\alpha) q_2^{-1-\alpha}  (2q_1-1)
\leq -(1-\alpha) (2 \cdot 0.75-1) <- 0.49,
\end{align*}
where we have again used the Mean Value Theorem, this time for the
function $z^{-\alpha}$, with $q_2 \in [1-q_1,q_1]$. 
We thus have $h(0.99)=h(0.75)+0.24h'(q_1)
\leq 2-0.24 \cdot 0.49 <1.9$, implying the desired result. 
\end{proof}

\section{Sparse families}\label{sec:sparse}

In this section, we study sparse families further, extending the known results in this regime.  In the first subsection, we prove Theorem~\ref{thm:afstability}, a stability result for Theorem~\ref{thm:af} of Alon and Frankl~\cite{af85}.  The second subsection contains the proof of Theorem~\ref{thm:sparse}, which sharpens the result of Daykin and Frankl~\cite{df83} by giving the asymptotics of the number of incomparable pairs in even sparser families.

\subsection{The proof of Theorem~\ref{thm:afstability}}

Following the original proof of Alon and Frankl, we shall determine the structure of a large set family with many comparable pairs by studying its comparability graph.  A \emph{comparability graph} $G_{\cF}$ of a set family $\cF$ has vertices $\cF$ and an edge $\{A, B\}$ if $A$ and $B$ are comparable.  Thus the number of edges in $G_{\cF}$ is precisely the number of comparable pairs $c(\cF)$.

The key to our proof of Theorem~\ref{thm:afstability} is the following lemma, which shows the comparability graph of such a large family cannot have many cliques of size $k+1$.  This lemma essentially appears in~\cite{af85} and~\cite{es11}, and we include its proof here for the sake of completeness.

\begin{lemma}
\label{lem:K_k+1-free}
For every $k \ge 2$ and $\gamma > 0$, there is an $n_0 = n_0(k, \gamma)$ such that for every $n \ge n_0$, if $\cF$ is a set family over $[n]$ of size $m = \card{\cF} \ge 2^{n/k}$, then $G_{\cF}$ contains at most $\gamma \binom{m}{k+1}$ copies of $K_{k+1}$.
\end{lemma}

To prove this lemma, we shall need a supersaturation result on complete multipartite graphs.  Let $K_r(\ell)$ denote the complete $r$-partite graph with $r$ parts of size $\ell$.  Given a graph $G$ with many $r$-cliques $K_r$, form an auxiliary $r$-uniform hypergraph $H$ over the same vertex set by placing an $r$-edge for each copy of $K_r$.  The complete multipartite $K_r(\ell)$ in $G$ then corresponds to a complete $r$-partite $r$-graph with parts of size $\ell$ in $H$.  Erd\H{o}s~\cite{erd64} showed such an $r$-partite $r$-graph must be contained in any sufficiently large $r$-graph of positive density, and a standard averaging argument gives the following proposition.

\begin{prop} \label{prop:blowups}
For integers $r \ge 2$, $\ell \ge 1$ and any real $\gamma > 0$, there is a $\delta_1 = \delta_1(r, \ell, \gamma)$ such that if $n$ is sufficiently large and $G$ is a graph on $n$ vertices with at least $\gamma \binom{n}{r}$ copies of $K_r$, then $G$ contains at least $\delta_1 \binom{n}{r \ell}$ copies of $K_r(\ell)$.
\end{prop}

We may now proceed with the proof of Lemma~\ref{lem:K_k+1-free}.  Our strategy is to show that a positive density of $(k+1)$-cliques implies a random sample of a large number of sets from $\cF$ must contain many sets from a low-dimensional subcube with positive probability.  On the other hand, a random sample from a family of this size gives rise to such a dense subcube with very low probability, giving the necessary contradiction.

\begin{proof}[Proof of Lemma~\ref{lem:K_k+1-free}]
Suppose for contradiction we have a family $\cF$ of size $m \ge 2^{n/k}$ with more than $\gamma \binom{m}{k+1}$ copies of $K_{k+1}$.  While we wish to have a positive density of cliques, it would be convenient to know there are not many larger cliques.  To this end, we define the recursive sequence
\[ \gamma_{k+1} = \gamma, \textrm{ and } \gamma_r = \frac{\delta_1\lt(r-1,3k(k+1),\gamma_{r-1} \rt)}{2 \binom{3k(k+1)(r-1)}{r}} \textrm{ for } k+2 \le r \le 12k^2, \]
where $\delta_1$ is as in Proposition~\ref{prop:blowups}.  Note that $\gamma_r > 0$ is independent of $n$ for each $k+1 \le r \le 12k^2$.  Let $t = \max\limits_{k+1 \le i \le 12k^2} \lt\{G_{\cF} \textrm{ has more than } \gamma_i \binom{m}{i} \textrm{ copies of } K_i \rt\}$, and note that if $t = 12k^2$, $G_{\cF}$ has more than $\gamma_t \binom{m}{t}$ copies of $K_{t}$.  Otherwise, if $t < 12k^2$, $G_{\cF}$ has more than $\gamma_t \binom{m}{t}$ copies of $K_t$ and at most $\gamma_{t+1} \binom{m}{t+1}$ copies of $K_{t+1}$.

\medskip

\noindent \underline{Case 1:}  $t = 12k^2$

\medskip

Let $S = \{ S_1, S_2, \hdots, S_t \}$ be a random sample of $t$ distinct vertices of $G_{\cF}$.  Since $G_{\cF}$ has more than $\gamma_t \binom{m}{t}$ cliques $K_t$, it follows that $G_{\cF}[S]$ is a clique on $t$ vertices with probability at least $\gamma_t$.  If $G_{\cF}[S]$ is complete, the corresponding sets in $\cF$ must form a $t$-chain $F_1 \subset \hdots \subset F_t \subset [n]$.  As $t=12k^2$, we can find some $j$ such that $\card{F_{j+6k}} - \card{F_j} \le n/(2k)$, which implies the $6k+1$ sets $F_j \subset F_{j+1} \subset \hdots \subset F_{j+6k}$ lie in a subcube of dimension $n/(2k)$.

Now there are $2^{n-n/(2k)} \binom{n}{n/(2k)}$ such subcubes, each containing $2^{n/(2k)}$ sets, and so the probability that a random sample of $t$ sets contains $6k$ sets from a single subcube of dimension $n/(2k)$ can be bounded above by
\[ \binom{t}{6k} \cdot 2^{n - \frac{n}{2k}} \binom{n}{\frac{n}{2k}} \lt( \frac{2^{\frac{n}{2k}}}{m} \rt)^{6k} \le 2^t \cdot 2^{2n} \lt( 2^{-\frac{n}{2k}} \rt)^{6k} = 2^t \cdot 2^{-n}, \]
since $m \ge 2^{n/k}$.  Since every sample producing a $t$-clique must give rise to such a dense subcube, we deduce $\gamma_t \le 2^t \cdot 2^{-n}$, which gives a contradiction when $n$ is suitably large.

\medskip

\noindent \underline{Case 2:}  $t < 12k^2$

\medskip

We now deal with the case when $t < 12k^2$, where we have more than $\gamma_t \binom{m}{t}$ copies of $K_t$, but at most $\gamma_{t+1} \binom{m}{t+1}$ copies of $K_{t+1}$.  Since the cliques we have may be relatively small, they will not be enough to find the dense subcubes we require.  However, we will find large complete multipartite graphs which, coupled with the lack of larger cliques, will give us a large number of sets contained in a tower of small cubes.

Indeed, letting $\ell = 3k(k+1)$, Proposition~\ref{prop:blowups} implies the density of $K_t(\ell)$ is at least $\delta_1 = \delta_1 \lt(t,\ell,\gamma_t \rt)$.  A random sample of $t \ell$ vertices of $G_{\cF}$ therefore gives a copy of $K_t(\ell)$ with probability at least $\delta_1$.  On the other hand, the density of $K_{t+1}$ is at most $\gamma_{t+1}$.  Thus the probability that our $t \ell$ random vertices contain some $t+1$ vertices inducing a $K_{t+1}$ is at most $\binom{t \ell}{t+1} \gamma_{t+1} \le \delta_1 / 2$.  Hence with probability at least $\delta_1 / 2$, a random sample of $t \ell$ vertices gives a $K_{t+1}$-free copy of $K_t(\ell)$.

If we consider the sets in $\cF$ corresponding to this copy of $K_t(\ell)$ in $G_{\cF}$, we find there are $t$ parts $\cF_i$, each consisting of $\ell$ sets $F_{i,j}$, such that for every $i \neq i'$ and $j,j'$, $F_{i,j} \subset F_{i',j'}$ or $F_{i',j'} \subset F_{i,j}$.  If we had some $i \neq i'$ and $j, j', j''$ such that $F_{i,j} \subset F_{i',j'} \subset F_{i,j''}$, then $F_{i,j}$ and $F_{i,j''}$ form a comparable pair within the part $\cF_i$, which would give a copy of $K_{t+1}$ in the corresponding subgraph of $G_{\cF}$, a contradiction.  Hence we may assume for each pair $i < i'$ and $j,j'$, we have $F_{i,j} \subset F_{i',j'}$.  Thus the sets are contained in the tower of $t$ cubes determined by the chain $\lt\{ \bigcup_{i \le p, 1 \le j \le \ell} F_{i,j} : 0 \le p \le t \rt\}$.  By considering the smallest cube, we find there must be $\ell$ sets from our sample in a subcube of dimension $n/t$.

Running the same probabilistic argument as before, the probability of finding $\ell$ sets from our sample within a subcube of dimension $n/t$ can be bounded above, giving the inequality
\[ \frac12 \delta_1 \le \binom{t \ell}{\ell} \cdot 2^{n - \frac{n}{t}} \binom{n}{\frac{n}{t}} \lt( \frac{2^{n/t}}{m} \rt)^{\ell} \le 2^{t \ell} \cdot 2^{2n} \lt( 2^{- \frac{n}{k(k+1)} } \rt)^{\ell} = 2^{t \ell} \cdot 2^{-n}, \]
since $t \ge k+1$ and $\ell = 3k(k+1)$.  This again provides a contradiction for $n$ sufficiently large, thus proving the lemma.
\end{proof}

This lemma shows set families of size $m \ge 2^{n/k}$ cannot have
many copies of $K_{k+1}$ in their comparability graphs.  We shall
now employ some well-known graph theoretic results to derive further
structural information for families with $c(\cF) \ge (1 - (1 + \eta)/k)
\binom{m}{2}$.  The first is the graph removal lemma, a generalisation
of the triangle removal lemma of Ruzsa--Szemer\'edi~\cite{rs76},
which shows
that we can make the comparability graph $K_{k+1}$-free by removing very
few edges.  For a detailed account of the history of the removal lemma,
see~\cite{cf13}.

\begin{theorem} \label{thm:removal}
For every fixed graph $H$ on $h$ vertices and $\eps > 0$ there is a $\delta > 0$ such that if $G$ is a graph on $m$ vertices with at most $\delta m^{h}$ copies of $H$, then $G$ can be made $H$-free by removing at most $\eps m^2$ edges.
\end{theorem}

In light of Lemma~\ref{lem:K_k+1-free}, this shows $G_{\cF}$ can be made $K_{k+1}$-free by removing at most $\eps m^2$ edges.  We would thus arrive at a $K_{k+1}$-free graph with at least $(1 - (1 + \eta)/k) \binom{m}{2} - \eps m^2$ edges, which is very close to the maximum possible.  Using stability for Tur\'an's theorem, we will deduce $G_{\cF}$ is very close in structure to the extremal Tur\'an graph $T_{m,k}$; that is, the balanced $k$-partite graphs.  Such a stability result was first proved by Simonovits~\cite{sim66}.  Here we use the following quantitative form of stability, whose short proof can be found in~\cite{ks05}.

\begin{theorem} \label{thm:turanstability}
Suppose $G$ is a $K_{k+1}$-free graph on $m$ vertices with at least $(1 - 1/k - \nu)\binom{m}{2}$ edges and $\nu < 1 / (4k)^4$.  Then there is a partition of the vertex set of $G$ as $V(G) = V_1 \cup V_2 \cup \hdots \cup V_k$ with $\sum e(U_i) < (2k+1) \nu^{1/2} m^2$.
\end{theorem}

Thus, after removing few edges from $G_{\cF}$ to make it $K_{k+1}$-free, we must be left with a graph that is very nearly $k$-partite.  To prove Theorem~\ref{thm:afstability}, it remains to transfer this structural information about $G_{\cF}$ to $\cF$ itself.

\begin{proof}[Proof of Theorem~\ref{thm:afstability}]
To simplify the presentation, we shall adopt asymptotic notation.  Let $\cF$ be a family over $[n]$ of size $m=\card{\cF} \ge (1 - o(1)) k 2^{n/k}$ with at least $\left( 1 - 1/k - o(1) \right) \binom{m}{2}$ comparable pairs.  As discussed above, we may apply Lemma~\ref{lem:K_k+1-free} and Theorems~\ref{thm:removal} and~\ref{thm:turanstability} to deduce that $G_{\cF}$ contains a $k$-partite subgraph with $(1 - 1/k - o(1)) \binom{m}{2}$ edges.

In other words, we can partition the sets of $\cF$ into $k$ families $\cF_1, \ldots,\cF_k$ such that for almost every $\{i,j\}\in \binom{[k]}{2}$ and almost every pair $(F_i,F_j)\in \cF_i\times \cF_j$, either $F_i\subset F_j$ or $F_j\subset F_i$, and almost all pairs inside each of the $k$ families are not comparable.  From the density of comparable pairs, it follows that each $\cF_i$ has size at least $(1 - o(1)) m/k > (1 - o(1)) 2^{n/k}$.

Now we claim that for every pair of indices $\{i,j\} \in \binom{[k]}{2}$,
either almost every set from $\cF_i$ is contained in almost every set from $\cF_j$ or vice versa.  Formally, there does not exist a constant $\beta>0$ such that for $\beta 2^{2n/k}$ pairs $(F_i,F_j)\in \cF_i\times \cF_j$ we have $F_i\subset F_j$, and for $\beta 2^{2n/k}$ such pairs $F_j\subset F_i$.  Indeed, the existence of such $\beta$ easily implies there exists a $\beta'>0$ and, without loss of generality, $\beta' 2^{3n/k}$ triples $(F_i,F_j,F_i')\in \cF_i\times \cF_j\times \cF_i$ such that $F_i\subset F_j\subset F_i'$.  However, this gives a positive density of comparable pairs $F_i \subset F_i'$ within $\cF_i$, contradicting our earlier observation.

Assume from now on that for every $i<j$, almost all sets from $\cF_i$ are contained in almost all sets from $\cF_j$.  We ``clean up'' the families by discarding any atypical sets, 
and so we may assume there is some arbitrarily small $\eps$ such that for every $i\in [k]$ we have subfamilies $\cF_i'\subseteq \cF_i$ of sizes at least $(1-\eps)2^{n/k}$ such that every set $F_i\in \cF_i'$ is contained in at least $(1-\eps)2^{n/k}$ sets from $\cF_{i+1}'$ (for $i<k$) and contains at least $(1-\eps)2^{n/k}$ sets from $\cF'_{i-1}$ (for $i>1$).

\medskip

We now generalise the proof of Lemma~\ref{lem:smalld}, seeking to discover a vertex partition corresponding to a tower of cubes containing $\bigcup_i \cF_i'$.  Recursively define the following partition of $[n]$: set $C_0=R_0=\emptyset$, and, for $i\in [k]$, let $C_i$ consist of all those elements from $[n]\setminus \bigcup_{j<i}\lt(C_j\cup R_j\rt)$ that appear in more than $\eps|\cF'_i|$ sets in $\cF'_i$,
and $R_i$ to be the set of all those elements from $[n]\setminus \bigcup_{j<i}\lt(C_j\cup R_j\rt)$
that appear in at least one and at most $\eps|\cF'_i|$ sets from $\cF'_i$.  Let $R_{k + 1} = [n] \setminus \bigcup_{j \le k} \lt( C_j \cup R_j \rt)$ complete the partition.

Now, for an arbitrary $i\in [k]$, let us take a closer look at the structure of the sets $F_i\in\cF'_i$.  
Note that if $x \in C_i \cup R_i$, there is some set $F_i \in \cF'_i$ such that $x \in F_i$.  Since $F_i$ is a subset of almost every set in $\cF'_{i+1}$, it follows that $x$ is contained in at least $(1 - \eps) \card{\cF'_{i+1}}$ sets in $\cF'_{i+1}$. Moreover, if $x$ appears in at least $\eps \card{\cF'_i}$ sets in $\cF'_{i}$, then it must appear in every set in $\cF'_{i+1}$, since these sets are comparable to at least $(1 - \eps) \card{\cF'_i}$ sets in $\cF'_i$.

From this, we can easily deduce that every set $F_i \in \cF'_i$ must contain $C_{i-1} \cup \bigcup_{j \le i-2} \lt( C_j \cup R_j \rt)$ as a subset.  Furthermore, every $x \in R_{i-1}$ must be contained it at least $(1 - \eps) \card{\cF'_i}$ sets in $\cF'_i$.  By definition of $R_i$, each element in $R_i$ is contained in at most $\eps \card{\cF'_i}$ sets in $\cF'_i$.  
Using Lemma~\ref{lem:entropy} we have
\[|\cF'_i|\leq 2^{|C_i|+H(\eps)(|R_{i-1}|+|R_i|)}.\]
Hence, we observe that, provided $\eps$ is sufficiently small, 
\[2^{n-0.5} < (1-\eps)^k 2^{n}\leq \prod_{i\in [k]}|\cF'_i|
\leq 2^{\sum_{i\in [k]}\lt(|C_i| + 2H(\eps)|R_i|\rt)} = 2^{n - (1-2H(\eps)) \sum_i \card{R_i} - \card{R_{k + 1}}} \le 2^{n - \sum_{i =1}^{k+1} \card{R_i}/2},\]
implying that for every $i\in [k + 1]$, $R_i=\emptyset$.  Moreover, since $\card{\cF_i'} \ge (1-\eps)2^{n/k}$, we must have $\card{C_i} \ge n/k$.  Since $\sum_{i \in [k]} \card{C_i} = n$, we in fact have equality.  In conclusion, the family $\cF'$, which $\cF$ is $\eps$-close to, is contained in the tower of cubes of dimension $n/k$ given by the chain $\lt\{ \bigcup_{j=0}^i C_j \rt\}_{0 \le i \le k}$.
\end{proof}

We conclude this subsection by deducing the exact result given in
Corollary~\ref{cor:exact}, showing that when $n$ is large, $k|n$ and $m =
k2^{n/k} - k + 1$, every extremal $m$-set family $\cF$ 
with $c(\cF) = c(n,m)$
is a tower of $k$ cubes of dimension $n/k$.

\begin{proof}[Proof of Corollary~\ref{cor:exact}]
Let $\ell = n / k$ be the dimensions of the cubes.  From Theorem~\ref{thm:afstability}, we know that all but $\eps \card{\cF}$ sets of $\cF$ are contained in a tower of cubes; without loss of generality, suppose these cubes correspond to the chain $\{[i \ell] : 0 \le i \le k \}$.  We show that if we have sets outside this tower of cubes, we can increase the number of comparable pairs by shifting them inside.

Suppose $G$ is a set in $\cF$ outside the tower of cubes, and let $i$ be such that $i\ell \le \card{G} \le (i + 1)\ell$.  Since $G$ is not in the tower of cubes, we cannot have $[i\ell] \subset G \subset [(i+1)\ell]$.  Consider how many sets from the subcubes between $[(i-1)\ell], [i\ell], [(i+1)\ell]$ and $[(i+2)\ell]$ $G$ can be comparable to.

In the subcube $\{ F : [(i-1)\ell ] \subset F \subset [i\ell] \}$, we can have at most $2^{\ell - \card{[i\ell] \setminus G}}$ sets with $[(i-1)\ell \subset F \subset [i\ell] \cap G$.  Similarly, in the subcube $\{F : [(i+1)\ell] \subset F \subset [(i+2)\ell] \}$, there are at most $2^{\ell - \card{G \setminus [(i+1)\ell] }}$ sets with $G \cup [(i+1)\ell] \subset F \subset [(i+2)\ell]$.

Finally we consider the subcube $\{ F : [i\ell] \subset F \subset [(i+1)\ell] \}$.  If $[i\ell] \subset G$, then there are at most $2^{\ell - \card{[(i+1)\ell] \setminus G}}$ sets $F$ with $[i\ell] \subset F \subset G \cap [(i+1)\ell]$; otherwise there are no such sets.  Similarly, if $G \subset [(i+1)\ell]$, there are at most $2^{\ell - \card{G \setminus [i\ell]}}$ sets $F$ with $[i\ell] \cup G \subset F \subset [(i+1)\ell]$; otherwise there are none.

Suppose $[i\ell] \not\subset G$ (the case $G \not\subset [(i+1)\ell]$ is similar).  $G$ can be comparable to at most $2^\ell + 2^{\ell - \card{[i\ell] \setminus G}} + 2^{\ell - \card{G \setminus [i\ell]}}$ sets from these three subcubes.  Since $\card{G} \ge i\ell$ and $[i\ell] \not\subset G$, it follows that $\card{G \setminus [i\ell]}, \card{[i\ell] \setminus G} \ge 1$.  If either one of these set differences has size at least $2$, then it follows that $G$ is comparable to at most $7\cdot 2^\ell / 4$ sets in these three subcubes.  Adding the remaining subcubes, and the small fraction of sets outside the tower of cubes, it follows that $G$ is in at most $\lt( k - 5/4 + \eps \rt) 2^\ell$ comparable pairs.  However, any set inside the tower of cubes would be in at least $\lt( k - 1 - \eps \rt) 2^\ell$ comparable pairs, and so we would increase the number of comparable pairs by replacing $G$ with a set missing from the tower of cubes.

Hence we may assume that $\card{G \setminus [i\ell]} = \card{[i\ell] \setminus G} = 1$ for every set $G$ outside the tower of cubes.  However, for each $i$, there are at most $n^2$ such sets, and hence we can have at most $n^3$ sets outside the tower of cubes.  Each such set is in at most $(k - 1)2^\ell + n^3$ nested pairs.  However, any set in the tower of cubes nests with $(k - 1)2^\ell$ sets from the other subcubes, and at least a further $2^{\ell/2 + 1} = 2^{n/(2 k) + 1} \gg n^3$ sets from its own subcube, and hence is in $(k - 1)2^\ell + 2^{\ell/2 + 1}$ nested pairs.  Thus it is again beneficial to replace sets outside the tower of cubes with sets within the tower.

This shows the extremal families must be towers of cubes, completing the proof.
\end{proof}

\subsection{The proof of Theorem~\ref{thm:sparse}}

We now turn to sparser families, where $\card{\cF} = m$ is subexponential (but superlinear).  Recall that in this setting we have $c(n,m) = (1 - o(1)) \binom{m}{2}$, so we instead count the number of incomparable pairs $i(n,m) = \binom{m}{2} - c(n,m)$.  We also parameterise the size of families by setting $m = n \ell$.

The key to proving Theorem~\ref{thm:sparse} is the following proposition, which gives the desired bound on $i(\cF)$ provided there are no sets in $\cF$ incomparable to many other sets.

\begin{prop} \label{prop:almostregular}
Given $\eps > 0$, for $\ell$ and $n$ sufficiently large we have $i(\cF) \ge \left( 1/2 - \eps \right) n \ell^2 \log \ell$ for any family $\cF$ of $n \ell$ sets over $[n]$ such that every set $F \in \cF$ is incomparable to at most $4 \ell \log \ell$ other sets.
\end{prop}

\begin{proof}
We shall obtain the desired result by finding dense subcubes of $\cF$, and applying Theorem~\ref{thm:af} to these subcubes.  We first define some parameters we shall need.  Set $k = \max \{ 81, 4 \eps^{-2} \}$, $s = k \log \ell$ and $r = n/s$.  Partition $[n]$ into $r$ intervals $I_1, I_2, \hdots, I_r$, each of length $s$, and partition $\cF = \cup_{j=1}^r \cF_j$, where $\cF_j = \{ F \in \cF : \card{F} \in I_j \}$.

We shall remove few sets from the subfamilies $\cF_j$ to obtain
cleaner $\cF_j' \subset \cF_j$ with the property that $\cF_j'$
is either empty or contained in a small subcube.  Call an interval $I_j$
\emph{light} if $\card{\cF_j} < \ell \sqrt{k} \log \ell$, and \emph{heavy}
otherwise.  If $I_j$ is light, we remove all the subsets from $\cF_j$
and set $\cF_j' = \emptyset$, thus losing at most $\ell \sqrt{k} \log
\ell$ sets.

On the other hand, if $I_j$ is heavy, let $F_{j,0}$ and $F_{j,1}$ be sets of minimum and maximum size in $\cF_j$ respectively.  By assumption, there are at most $8 \ell \log \ell$ sets in $\cF_j$ that are incomparable to one of $F_{j,0}$ or $F_{j,1}$.  Let $\cF_j'$ be the remaining subfamily, and observe we must have $F_{j,0} \subseteq F \subseteq F_{j,1}$ for any $F \in \cF_j'$.  Since $I_j$ is heavy, we have $\card{\cF_j'} \ge \ell \sqrt{k} \log \ell - 8 \ell \log \ell \ge \ell$, since $\sqrt{k} \ge 9$.  Moreover, $\cF_j'$ is contained in the subcube spanned by $F_{j,0}$ and $F_{j,1}$, which has dimension $\card{F_{j,1}} - \card{F_{j,0}} \le \card{I_j} = s$.

Since $\cF_j'$ lies in a subcube of dimension $s$, we shall use Theorem~\ref{thm:af} to estimate $i(\cF_j')$.  Setting $\delta = 1/(k(k+1))$, we have $2^{(1/(k+1) + \delta)s} = 2^{s/k} = \ell \le \card{\cF_j'}$.  Hence, by Theorem~\ref{thm:af}, we have, for some $\beta = \beta_k$ and constant $C$, 
\[ i(\cF_j') \ge \frac{1}{k} \binom{\card{\cF_j'}}{2} - C \card{\cF_j'}^{2 - \beta \delta^{k+1}} \ge \left( 1 - \frac{\eps}{2} \right) \frac{1}{k} \binom{\card{\cF_j'}}{2}, \]
provided we take $\ell \ge (4 C k \eps^{-1} )^{(k(k+1))^{k+1} / \beta}$.

Finally, in passing to these cleaner subfamilies $\cF_j'$, we lost at most $\ell \sqrt{k} \log \ell$ sets from each $\cF_j$.  Hence we have $\sum_{j=1}^r \card{\cF_j'} \ge \card{\cF} - r \ell \sqrt{k} \log \ell = n \ell - n \ell \sqrt{k} \log \ell / s = (1 - k^{-1/2}) n \ell \ge (1 - \eps / 2) n \ell$.  With these estimates in place, we can now lower bound $i(\cF)$ by summing up $i(\cF_j')$ over each interval.  This gives
\begin{align*}
i(\cF) &\ge \sum_{j=1}^r i(\cF_j') \ge \left(1 - \frac{\eps}{2} \right) \frac{1}{k} \sum_{j=1}^r \binom{\card{\cF_j'}}{2} \ge (1 - \eps) \frac{1}{2k} \sum_{j=1}^r \card{\cF_j'}^2 \\
	&\ge (1 - \eps) \frac{1}{2rk} \left( \sum_{j=1}^r \card{\cF_j'} \right)^2 \ge (1 - \eps) \left(1 - \frac{\eps}{2} \right)^2 \frac{n^2 \ell^2}{2rk} \ge \left( \frac12 - \eps \right) n \ell^2 \log \ell,
\end{align*}
where we use the Cauchy-Schwarz inequality in the second line.  Hence we obtain the claimed bound on $i(\cF)$.
\end{proof}

We now deduce Theorem~\ref{thm:sparse} from the above proposition.

\begin{proof}[Proof of Theorem~\ref{thm:sparse}]
Given $\eps > 0$, let $\ell$ be twice as large as needed in Proposition~\ref{prop:almostregular}.  We shall show that if $\cF$ is a family of $n \ell$ sets over $[n]$, then $i(\cF) \ge \left(1/2 - \eps\right) n \ell^2 \log \ell$.  To do so, we successively remove sets from $\cF$ until no set is incomparable to many other sets, and then apply the proposition.

Start with $\cF_0 = \cF$ and $\ell_0 = \ell$.  Now, given $\cF_j$ and $\ell_j$, if there is a set $F_j \in \cF_j$ incomparable to at least $4 \ell_j \log \ell_j$ other sets in $\cF_j$, then define $\cF_{j+1} = \cF_j \setminus \{ F_j \}$ and set $\ell_{j+1} = \ell_j - 1/n$, so that we have $\card{\cF_{j+1}} = n \ell_{j+1}$.  Otherwise, stop the process and set $t = j$.

Since the set $F_j$ accounts for at least $4 \ell_j \log \ell_j$ distinct incomparable pairs, we have
\[ i(\cF) \ge 4 \sum_{j=0}^{t-1} \ell_j \log \ell_j + i(\cF_t). \]

Observe that $\ell_j \ge \ell / 2$ for $j \le n \ell / 2$.  Hence, if $t \ge n \ell /2$, we have
\[ i(\cF) \ge 4 \sum_{j=0}^{n \ell / 2} \ell_j \log \ell_j \ge n \ell^2 \log \frac{\ell}{2} > \frac12 n \ell^2 \log \ell, \]
as required.  Hence we may assume $t \le n \ell / 2$, and thus $\ell_t \ge \ell / 2$.  By our bound on $\ell$, this implies that we may apply Proposition~\ref{prop:almostregular} to $\cF_t$, giving
\begin{equation} \label{eqn:telescope}
i(\cF) \ge 4 \sum_{j=0}^{t-1} \ell_j \log \ell_j + i(\cF_t) \ge 4 \sum_{j=0}^{t-1} \ell_j \log \ell_j + \left( \frac12 - \eps \right) n \ell_t^2 \log \ell_t.
\end{equation}

Since $\ell_{j+1} = \ell_j - 1/n$ and $- \log (1 - x) \le \log(1 + 2x) \le 4 x$ for $0 \le x \le 1/2$, we have
\begin{align*}
\left( \frac12 - \eps \right)&n \ell_j^2 \log \ell_j - \left( \frac12 - \eps \right) n \ell_{j+1}^2 \log \ell_{j+1} \\
&= \left( \frac12 - \eps \right) \left( n \ell_j^2 \log \ell_j - n \left( \ell_j - \frac{1}{n} \right)^2 \log \ell_j - n \left( \ell_j - \frac{1}{n} \right)^2 \log \left( 1 - \frac{1}{n \ell_j} \right) \right) \\
&\le \left( \frac12 - \eps \right) \left( 2 \ell_j \log \ell_j + 4 \ell_j \right) < 4 \ell_j \log \ell_j,
\end{align*}
which we may use to telescope the sum in~\eqref{eqn:telescope} and deduce that
\[ i(\cF) \ge 4 \sum_{j=0}^{t-1} \ell_j \log \ell_j + \left( \frac12 - \eps \right) n \ell_t^2 \log \ell_t \ge \left( \frac12 - \eps \right) n \ell_0^2 \log \ell_0 = \left( \frac12 - \eps \right) n \ell^2 \log \ell. \]

Since the family $\cF$ was arbitrary, it follows that $i(n,n\ell) \ge ( 1/2 - \eps ) n \ell^2 \log \ell$.
\end{proof}

\section{Dense families} \label{sec:dense}

In this section we consider much denser families, for which $m \ge 2^{0.92n}$.  Theorem~\ref{thm:af} shows that such families have $o(m^2)$ comparable pairs, and so the situation is qualitatively different from the preceding sections.  The extremal families are of a different nature as well.  We shall prove Theorem~\ref{thm:dense}, which shows that the extremal families contain sets as far from the middle levels as possible.  To this end, it is more convenient to work with the complementary set family $\cG = 2^{[n]} \setminus \cF$.  Note that to maximise $c(\cF)$, we must minimise the number of comparable pairs containing at least one set from $\cG$.  Since sets in the middle levels are contained in the fewest comparable pairs, it is intuitive that $\cG$ should be taken from the middle levels, and we will use shifting arguments to make this intuition precise.

\begin{proof}[Proof of Theorem~\ref{thm:dense}]
Let $\cF \subset 2^{[n]}$ be a family of $m$ sets maximising $c(\cF)$.  We consider the complementary family $\cG = 2^{[n]} \setminus \cF$, which minimises the number of comparable pairs containing a set from $\cG$.  We need to show that $\cG$ contains all sets $F$ with $k+1 \le \card{F} \le n - k -1$ and does not contain any set $F$ with $\card{F} \le k-1$ or $\card{F} \ge n-k + 1$.

To this end, we partition $2^{[n]}$ into three families.  Let $\cA_0 = 2^{[n]} \setminus \cH_{k-1}$ be the family of all sets of sizes between $k$ and $n-k$, $\cA_1 = \cH_{k-1} \setminus \cH_{k-2}$ the family of all sets of size $k-1$ or $n-k+1$, and $\cA_2 = \cH_{k-2}$ the family of all sets of size at most $k-2$ or at least $n-k+2$.  For $0 \le i \le 2$, define $\cG_i = \cG \cap \cA_i$.

We start by showing that $\cG \subset \cA_0$, or, equivalently, $\cG_1 \cup \cG_2 = \emptyset$.  Suppose we had $m' = \card{\cG_1} + \card{\cG_2} > 0$.  By shifting the sets in $\cG_1 \cup \cG_2$ into $\cA_0$, we shall decrease the number of comparable pairs involving sets in $\cG$, contradicting the optimality of $\cG$ (and hence of $\cF$).

First note that since $\card{\cF} \ge M_{k-1} = \card{\cH_{k-1}}$, $\card{\cG} \le \card{\cA_0}$, and hence there must be at least $m'$ available sets in $\cA_0 \setminus \cG$ to shift to.  Each such set is comparable to at most $2^{n-k} + 2^k - 2 = \left( 1 + 2^{-(n-2k)} \right) 2^{n-k} - 2$ other sets, and hence the number of new comparable pairs the shift could introduce is at most $\left[ \left(1 + 2^{-(n-2k)} \right) 2^{n-k} - 2 \right] m'$.

On the other hand, we bound from below the number of comparable pairs we would lose through this shift.  We would lose all the pairs containing sets in $\cG_1 \cup \cG_2$ except those that also contain a set from $\cG_0$.  To avoid counting the latter pairs, for every set $G \in \cG_1 \cup \cG_2$ we shall count only the sets in $\cA_2$ comparable to $G$.  In this way we do not count any pairs containing a set from $\cG_0$.  Furthermore, a pair containing a set from $\cG_1$ is counted only once.  However, comparable pairs with both sets from $\cG_2$ are counted twice, which we shall have to account for.

The following lemma, proven later, will be required for our calculations.

\begin{lemma} \label{lem:binomtail}
If $\lambda \ge 0$ and $\cA = \{ A \subset [r] : \card{A} \ge r/2 + \lambda \sqrt{r} \}$, then $\card{\cA} \le e^{- 2 \lambda^2} 2^{r}$.
\end{lemma}

We first consider sets $G \in \cG_1$.  By symmetry, we may assume $\card{G} = n-k+1$.  $G$ is contained in $2^{k-1} - 1$ larger sets in $\cA_2$.  The number of subsets of $G$ in $\cA_2$ is given by $\sum_{j=0}^{k-2} \binom{n-k+1}{j} = 2^{n-k+1} - \sum_{j = k-1}^{n-k+1} \binom{n-k+1}{j}$.  By our bound on $k$, it follows that $2(k-1) > (n-k+1) + 2\sqrt{2 n \ln 2}$, and so we may apply Lemma~\ref{lem:binomtail} with $r = n-k+1$ and $\lambda = \sqrt{2 \ln 2}$ to deduce
\[ 2^{n-k+1} - \sum_{j=k-1}^{n-k+1} \binom{n-k+1}{j} \ge \left( 1 - e^{- 4 \ln 2} \right) 2^{n-k+1} = \frac{15}{16} \cdot 2^{n-k+1}. \]
Hence, when shifting, we lose at least $15 \cdot 2^{n-k+1} / 16 + 2^{k-1} - 1 = \left( 15/8 + 2^{-(n-2k+1)} \right) 2^{n-k} - 1$ comparable pairs for each $G \in \cG_1$.

We now perform a similar analysis for $G \in \cG_2$.  Again, by symmetry, we may assume $\card{G} \ge n-k+2$.  Such a set is contained in $2^{n- \card{G}} - 1$ larger sets in $\cA_2$.  The number of subsets of $G$ in $\cA_2$ is at least $\sum_{j=0}^{k-2} \binom{\card{G}}{j}$.  The following lemma, whose proof we defer for the moment, shows that this quantity is minimised when $\card{G}$ is as small as possible.

\begin{lemma} \label{lem:optimisation}
Given integers $n$ and $s$ satisfying $n/3 < s \le n/2$, the quantity
\[ 2^{n-t} + \sum_{j \le s} \binom{t}{j} \]
is minimised over $n-s \le t \le n$ when $t = n-s$.
\end{lemma}

Hence, for the worst-case scenario, we may assume $\card{G} = n-k+2$.  Using Lemma~\ref{lem:binomtail} with $r = n - k + 2$ and $\lambda = \sqrt{2 \ln 2}$, we may similarly conclude that each set $G \in \cG_2$ contributes at least $15 \cdot 2^{n-k+2} / 16 + 2^{k-2} - 1= (15/4 + 2^{-(n-2k+2)}) 2^{n-k} - 1$ comparable pairs.

Recalling that comparable pairs between sets in $\cG_2$ are double-counted, it follows that by shifting the sets in $\cG_1 \cup \cG_2$ to $\cA_0$, we lose at least
\begin{align*}
\left[ \left( \frac{15}{8} + 2^{-(n-2k+1)} \right) 2^{n-k} - 1 \right] \card{\cG_1} + \frac12 &\left[ \left( \frac{15}{4} + 2^{-(n-2k+2)} \right) 2^{n-k} - 1 \right] \card{\cG_2} \\
	> \left[ \left( \frac{15}{8} + 2^{-(n-2k+3)} \right) 2^{n-k} - 1 \right] & m'
\end{align*}
comparable pairs, while we gain at most $\left[\left(1 + 2^{-(n-2k)} \right) 2^{n-k} - 2 \right] m'$ new pairs.  Since $(15/8 + 2^{-(n-2k+3)}) - (1 + 2^{-(n-2k)}) = 7(1 - 2^{-(n-2k)})/8 \ge 0$, it follows that shifting the sets from $\cG_1 \cup \cG_2$ decreases the number of comparable pairs involving sets in $\cG$.  This gives the desired contradiction, and hence we must have $\cG \subset \cA_0$.

It remains to show that $\cG$ contains all sets of sizes between $k+1$ and $n-k-1$.  Note that this is trivial unless $n \ge 2k+2$.  Suppose for contradiction there is some set $G_0 \notin \cG$ with $k+1 \le \card{G_0} \le n-k-1$.  Given the size of $\cG$, there must be some $G_1 \in \cG$ with $\card{G_1} \in \{k,n-k\}$.  We use the same shifting arguments as before to deduce that we should replace $G_1$ with $G_0$.

Indeed, adding $G_0$ to $\cG$ can introduce at most $2^{n-k-1} + 2^{k+1} - 2 = \left( 1 + 2^{-(n-2k-2)} \right) 2^{n-k-1} - 2$ new comparable pairs.  On the other hand, since $\cG \subset \cA_0$, by removing $G_1$ we would lose all comparable pairs between $G_1$ and $\cA_1 \cup \cA_2$.  As above, we apply Lemma~\ref{lem:binomtail}, this time with $r = n-k$ and $\lambda = \sqrt{2 \ln 2}$, to find this gives at least $\left( 15/8 + 2^{-(n-2k-1)} \right) 2^{n-k-1} - 1$ comparable pairs.  Since $n \ge 2k+2$, $15/8 + 2^{-(n-2k-1)} > 1 + 2^{-(n-2k-2)}$, and so switching $G_1$ for $G_0$ decreases the number of comparable pairs containing a set from $\cG$, contradicting the optimality of $\cG$.

By taking complements, it follows that any optimal family $\cF$ must satisfy $\cH_{k-1} \subset \cF \subset \cH_k$, as claimed.
\end{proof}

We now prove the two lemmas used in the proof above.

\begin{proof} [Proof of Lemma~\ref{lem:binomtail}]
For every $i \in [r]$, we have $\card{\cA(i)} \ge \left(1/2 + \lambda / \sqrt{r} \right) \card{\cA}$.  Applying Lemma~\ref{lem:entropy} gives $\card{\cA} \le 2^{r H(1/2 +  \lambda r^{-1/2})}$.  The second-order Taylor series expansion of $H(x) = -x \log x - (1-x) \log (1-x)$ gives $H(1/2 + q) = H(1/2) + q H'(1/2) + q^2 H''(z)/2$ for some $z$ between $1/2$ and $1/2 + q$.  Since $H(1/2) = 1$, $H'(1/2) = 0$, and $H''(z) = - (z^{-1} + (1-z)^{-1})/ \ln 2 \le - 4 / \ln 2$, we have the bound $H(1/2 + q) \le 1 - 2q^2 / \ln 2$, and so
\[ \card{\cA} \le 2^{r H(\frac12 + \lambda r^{-1/2})} \le 2^{r(1 - 2 \lambda^2 / (r \ln 2))} = e^{-2 \lambda^2} 2^r. \]
\end{proof}

\begin{proof} [Proof of Lemma~\ref{lem:optimisation}]
We show that the quantity $x_t = 2^{n-t} + \sum_{j \le s} \binom{t}{j}$ is increasing in $t$.  Recalling that $t \ge n-s$, we have
\begin{align*}
	x_{t+1} - x_t &= \left( 2^{n-(t+1)} + \sum_{j \le s} \binom{t+1}{j} \right) - \left( 2^{n-t} + \sum_{j \le s} \binom{t}{j} \right) \\
	&= 2^{n-t-1} - 2^{n-t} + \sum_{j \le s} \left( \binom{t+1}{j} - \binom{t}{j} \right) = \sum_{j \le s} \binom{t}{j-1} - 2^{n-t-1} \\
	&= \sum_{j \le s-1} \binom{t}{j} - 2^{n-t-1} \ge \sum_{j \le s-1} \binom{n-s}{j} - 2^{s-1} = 2^{n-s} - \sum_{j = s}^{n-s} \binom{n-s}{j} - 2^{s-1}.
\end{align*}

Since $s > n/3$, we have $s > (n-s)/2$, and so the sum of the binomial coefficients is at most $2^{n-s-1}$.  Thus we have $x_{t+1} - x_t \ge 2^{n-s-1} - 2^{s-1} \ge 0$, since $s \le n/2$.  Hence $x_t$ is minimised when $t = n-s$.
\end{proof}

This completes the proof of Theorem~\ref{thm:dense}, which gives the large-scale structure of dense extremal families.  If $M_{k-1} \le m \le M_k$, and $\cF$ is a family of size $m$ with the maximum number of comparable pairs, then we know $\cH_{k-1} \subset \cF \subset \cH_k$.  This is enough to determine $c(n,m)$ asymptotically, but if we wish to find the exact value of $c(n,m)$, we must determine which sets of sizes $k$ or $n-k$ should be contained in $\cF$.

As the following corollary shows, this is easy for some particular ranges of $m$.  To define the optimal construction, we shall assume for simplicity that $n$ is even.  Fix an arbitrary subset $X$ of $n/2$ elements from $[n]$.  Given $m = M_{k-1} + m'$, with $0 \le m' \le 2 \binom{n/2}{k}$, let $\cF_{m}^{\ast} = \cH_{k-1} \cup \cA \cup \cB$, where $\cA$ is a set of $\floor{m'/2}$ $k$-subsets of $X$, while $\cB$ is a set of $\ceil{m'/2}$ $(n-k)$-sets containing $X$.

\begin{cor} \label{cor:exactdense}
Suppose $n/3 + \sqrt{2 n \ln 2} \le k < n/2$ and $m = M_{k-1} + m'$, where $0 \le m' \le 2 \binom{n/2}{k}$.  Then $c(n,m) = c(\cF_m^{\ast})$.
\end{cor}

\begin{proof}
By Theorem~\ref{thm:dense}, we know that we can partition $\cF = \cH_{k-1} \cup \cA \cup \cB$, where $\cA$ is a family of $k$-sets and $\cB$ is a family of $(n-k)$-sets.  Let $a = \card{\cA}$ and $b = \card{\cB}$; we must have $a + b = m'$.

By symmetry, each set in $\cA \cup \cB$ has the same number of comparable pairs with sets in $\cH_{k-1}$, and so $c(\cF)$ is maximised precisely when $c(\cA, \cB)$ is maximised.  We trivially have $c(\cA, \cB) \le ab \le \floor{(m')^2 / 4}$.  On the other hand, in our family $\cF_m^{\ast}$ every $A \in \cA$ and $B \in \cB$ satisfy $A \subset X \subset B$, and so we indeed have $c(\cA, \cB) = ab = \floor{(m')^2 /4}$.  Thus $\cF_m^{\ast}$ maximises the number of comparable pairs.
\end{proof}

\section{Concluding remarks} \label{sec:conc}

In this paper we studied the maximum number of comparable pairs present in families of a given size, continuing the work of Alon, Daykin, Frankl and others.  One feature that sets this problem apart from a number of others in extremal set theory is the absence of a nested sequence of extremal families.  We have shown that when $m = k2^{n/k} - k + 1$, a tower of $k$ cubes is optimal, while when $m \ge 2^{0.92n}$, the extremal families avoid the middle layers, and hence there cannot be a sequence of families $\cF_0 \subset \cF_1 \subset \hdots \subset \cF_{2^n}$ such that $\card{\cF_m} = m$ and $c(\cF_m) = c(n,m)$.  This precludes the global use of many standard techniques in the field, thus partially explaining the different arguments needed for the sparse and dense regimes.  It also leaves a number of open problems, some of which we discuss below.

\subsection*{The Alon--Frankl conjecture}

In Section~\ref{sec:afconj}, we settled a conjecture of Alon and Frankl, showing $m = n^{\omega(1)} 2^{n/2}$ implies $c(n,m) = o(m^2)$.  More precisely, we showed that if $\cA$ and $\cB$ are two families over $[n]$ with $\card{\cA} \card{\cB} = n^d 2^n$, then $c(\cA, \cB) \le 2^{-d/300} \card{\cA}\card{\cB}$.  On the other hand, the construction of Alon and Frankl shows that we can have $\card{\cA} \card{\cB} = \Omega( n^d 2^n )$ and $c(\cA, \cB) \ge 2^{-d} \card{\cA}\card{\cB}$.  It would be very interesting to close this gap, and determine the true constant in the exponent.  An exhaustive search for $n \le 7$ suggests that the answer lies much closer to $1$, and that the construction may be near-optimal.

\subsection*{Sparse families}

We have shown that towers of finitely many cubes are extremal families,
and that they are asymptotically extremal when the number of cubes tends
to infinity (sublinearly in $n$).  However, these towers of cubes can
only exist for certain family sizes, and it remains to understand what the
extremal families are in between.  For instance, how do the extremal
families transition from a tower of two cubes to a tower of three cubes
when $m$ ranges from $2\cdot 2^{n/2} - 1$ to $3 \cdot 2^{n/3} - 2$?

\subsection*{Dense families}

Finally, in the dense setting we determined the approximate structure of the extremal families.  Thus the problem of determining $c(n,m)$ in this range reduces to maximising $c(\cA, \cB)$ over $\cA \subset \binom{[n]}{k}$ and $\cB \subset \binom{[n]}{n-k}$ for fixed $m' = \card{\cA} + \card{\cB}$.

We showed in Corollary~\ref{cor:exactdense} that one can achieve the trivial upper bound $c(\cA,\cB) = \floor{(m')^2/4}$ when $m' \le 2\binom{n/2}{k}$.  This is in fact the largest $m'$ for which this upper bound can be attained, as it is only possible when $\card{\cA}, \card{\cB} \ge \floor{m'/2}$ and $c(\cA, \cB) = \card{\cA}\card{\cB}$.  The latter implies there is some $X \subset [n]$ with $A \subset X \subset B$ for all $A \in \cA$ and $B \in \cB$, which requires one of the two families to have size at most $\binom{n/2}{k}$.

One might slightly generalise the above two-level problem to obtain the following question.

\begin{question} \label{ques:general}
Given $0 \le k_1 \le k_2 \le n$, $0 \le a \le \binom{n}{k_1}$ and $0 \le b \le \binom{n}{k_2}$, which families $\cA \subset \binom{[n]}{k_1}$ and $\cB \subset \binom{[n]}{k_2}$ maximise $c(\cA, \cB)$?
\end{question}

While this question may seem a mild generalisation of our original
problem, it has some subtle complexities.  In particular, it contains
as a special case the famous Kruskal--Katona theorem~\cite{kat68, kru63}.
Let $k_1 = k_2 - 1$.  Note that each set in $\binom{[n]}{k_2}$ contains
$k_2$ sets in $\binom{[n]}{k_1}$.  The minimum size of a lower shadow
of $b$ sets of size $k_2$ is then the minimum $a$ such that the solution
to the above problem is $b k_2$.  One can similarly model the upper shadows.

By Kruskal--Katona, we know the lower shadows are minimised when $\cA$ and $\cB$ are initial segments of the colexicographic ordering, while the upper shadows are minimised by initial segments of the lexicographic ordering.  On the other hand, in Corollary~\ref{cor:exactdense} the extremal families were mixed: initial segments of the lexicographic ordering for $\cA$ and colexicographic ordering for $\cB$.  This shows that, once again, there will not be a nested sequence of solutions, and the nature of the extremal families depends on the range of parameters.

\subsection*{The phase transition}

As we have mentioned before, the extremal families are very different in nature for small and large $m$.  Our results give information about $c(n,m)$ when either $m \le n^{d}2^{n/2}$ or $m \ge 2^{0.92n}$.  It would be of great interest to determine what happens in between, and to examine how the extremal families transition from tower-of-cubes-type of constructions to those consisting of sets of size far from $n/2$.


\begin{thebibliography}{99}

\bibitem{af85}
N.~Alon and P.~Frankl,
\newblock{\em The maximum number of disjoint pairs in a family of subsets},
\newblock{Graph. Combinator.} {\bf 1} (1985), 13--21.

\bibitem{as}
N.~Alon and J.~Spencer,
\newblock{\em The Probabilistic Method}, Third Edition,
\newblock{John Wiley Inc., New York} (2008).

\bibitem{cf13}
D.~Conlon and J.~Fox,
\newblock{\em Graph removal lemmas},
\newblock{Surveys in Combinatorics, Cambridge University Press} (2013), 1--50.

\bibitem{dgs14}
S.~Das, W.~Gan and B.~Sudakov,
\newblock{\em Sperner's theorem and a problem of Erd\H{o}s--Katona--Kleitman},
\newblock{Combin. Probab. Comput.}, to appear.

\bibitem{df83}
D.E.~Daykin and P.~Frankl,
\newblock{\em On Kruskal's cascades and counting containments in a set of subsets},
\newblock{Mathematika} {\bf 30} (1983), 133--141.

\bibitem{dgks14}
A.P.~Dove, J.R.~Griggs, R.J.~Kang and J.-S.~Sereni,
\newblock{\em Supersaturation in the Boolean lattice},
\newblock{Integers} {\bf 14A} (2014), 1--7.

\bibitem{es11}
D.~Ellis and B.~Sudakov,
\newblock{\em Generating all subsets of a finite set with disjoint unions},
\newblock{J. Comb. Theory A} {\bf 118} (2011), 2319--2345.

\bibitem{erd45}
P.~Erd\H{o}s,
\newblock{\em On a lemma of Littlewood and Offord},
\newblock{Bull. Am. Math. Soc.} {\bf 51} (1945), 898--902.

\bibitem{erd64}
P.~Erd\H{o}s,
\newblock{\em On extremal problems of graphs and generalized graphs},
\newblock{Israel J. Math.} {\bf 2.3} (1964), 183--190.

\bibitem{erd81}
P.~Erd\H{o}s,
\newblock{\em Problem sessions},
\newblock{In: Ordered Sets (Proc. NATO Adv. Study)},
\newblock{edited by I.~Rival}, Dordrecht: Reidel (1981), 860--861.

\bibitem{guy83}
R.K.~Guy,
\newblock{\em A miscellany of Erd\H{o}s problems},
\newblock{Am. Math. Mon.} {\bf 90} (1983), 118--120.

\bibitem{kat68}
G.O.H.~Katona,
\newblock{\em A theorem of finite sets},
\newblock{in: Theory of Graphs, ed: P.~Erd\H{o}s, G.~Katona}, Akademia Kiad\'o, Budapest, 1966, 187--207.

\bibitem{ks05}
P.~Keevash and B.~Sudakov,
\newblock{\em On a hypergraph Tur\'an problem of Frankl},
\newblock{Combinatorica} {\bf 25.6} (2005), 673--706.

\bibitem{kle66}
D.~Kleitman,
\newblock{\em A conjecture of Erd\H{o}s--Katona on commensurable pairs among subsets of an $n$-set},
\newblock{Theory of Graphs, Proc. Colloq., Tihany} (1966), 215--218.

\bibitem{kru63}
J.B.~Kruskal,
\newblock{\em The number of simplices in a complex},
\newblock{in: Mathematical Optimization Techniques, ed: R.~Bellman}, University of California Press, 1963, 251--278.

\bibitem{rs76}
I.Z.~Ruzsa and E.~Szemer\'edi,
\newblock{\em Triple systems with no six points carrying three triangles},
\newblock{in Combinatorics (Keszthely, 1976), Coll. Math. Soc. J. Bolyai} {\bf 18}, Volume II, 939--945.

\bibitem{sim66}
M.~Simonovits,
\newblock{\em A method for solving extremal problems in graph theory, stability problems;}
\newblock{in Theory of Graphs (Proc. Colloq. Tihany, 1966), Academic Press, New York, and Akad. Kiad\'o, Budapest} (1968), 279--319.

\bibitem{spe28}
E.~Sperner,
\newblock{\em Ein Satz \"uber Untermengen einer endlichen Menge},
\newblock{Mathematische Zeitschrift} {\bf 27} (1928), 544--548.

\end{thebibliography}
\end{document}